\documentclass[12pt]{article}
\usepackage[utf8]{inputenc}
\usepackage[T1]{fontenc}
\usepackage{array}
\usepackage{amsmath, amsthm}
\usepackage{amssymb}
\usepackage{enumitem}

\newtheorem{theorem}{Theorem}[section]
\newtheorem{proposition}[theorem]{Proposition}
\newtheorem{corollary}[theorem]{Corollary}
\newtheorem{conjecture}[theorem]{Conjecture}
\newtheorem{lemma}[theorem]{Lemma}

\newtheorem{definition}[theorem]{Definition}

\newtheorem{remark}[theorem]{Remark}

\DeclareMathOperator{\rwc}{R} %rich word count
\DeclareMathOperator{\rws}{RW} %rich word set
\DeclareMathOperator{\Alpha}{A}
\DeclareMathOperator{\Factor}{Fac}
\DeclareMathOperator{\Pal}{Pal}
\DeclareMathOperator{\dd}{g} %constant 
\DeclareMathOperator{\PL}{PL} %palindromic length
\DeclareMathOperator{\LUF}{luf} %length of ups factorization

\begin{document}

\title{Palindromic factorization of rich words}

\author{Josef Rukavicka\thanks{Department of Mathematics,
Faculty of Nuclear Sciences and Physical Engineering, Czech Technical University in Prague
(josef.rukavicka@seznam.cz).}}

\date{\small{October 25, 2021}}

\maketitle

\begin{abstract}
A finite word $w$ is called \emph{rich} if it contains $\vert w\vert+1$ distinct palindromic factors including the empty word. For every finite rich word $w$ there are distinct nonempty palindromes $w_1, w_2,\dots,w_p$ such that $w=w_pw_{p-1}\cdots w_1$ and $w_i$ is the longest palindromic suffix of $w_pw_{p-1}\cdots w_i$, where $1\leq i\leq p$. This palindromic factorization is called \emph{UPS-factorization}. Let $\LUF(w)=p$ be \emph{the length of UPS-factorization} of $w$. 

In 2017, it was proved that there is a constant $c$ such that if $w$ is a finite rich word and $n=\vert w\vert$ then $\LUF(w)\leq c\frac{n}{\ln{n}}$.
We improve this result as follows: There are constants $\mu, \pi$ such that if $w$ is a finite rich word and $n=\vert w\vert$ then \[\LUF(w)\leq \mu\frac{n}{e^{\pi\sqrt{\ln{n}}}}\mbox{.}\]
The constants $c,\mu,\pi$ depend on the size of the alphabet.
\end{abstract}

\section{Introduction}
A \emph{palindrome} is a word that is equal to its reversal; for example ``radar'' or ``noon''. A word $w$ is called \emph{rich} if it contains $\vert w\vert+1$ distinct palindromic factors  including the empty word. It is known that a word $w$ cannot have more than $\vert w\vert+1$ distinct palindromic factors \cite{DrJuPi}, thus less formally said,  rich words are those words that contain the maximal number of distinct palindromic factors. The rich words have attracted quite a lot of attention over the last couple of decades; see, for instance, \cite{AGO2021184, BuLuGlZa2, DrJuPi, GlJuWiZa, RUKAVICKA2021richext, Vesti2014}. A well known example of rich words are Sturmian words \cite{BaPeSta2}.

In \cite[Definition $4$ and Proposition $3$]{DrJuPi} it was proved that the longest palindromic suffix of a rich word $w$ has exactly one occurrence in $w$. Using this property it was proved that:

\begin{lemma} (see \cite[Lemma $1$]{RukavickaRichWords2017}) 
\label{lemma_A}
Let $w$ be a finite nonempty rich word. There exist distinct nonempty palindromes $w_1,w_2,\dots,w_p$ such that
\[\begin{split} 
w=w_pw_{p-1}\cdots w_1
 \mbox{ and }w_i \mbox{ is the longest palindromic suffix of }\\ w_pw_{p-1}\cdots w_i \mbox{ for }i=1,2,\dots ,p\mbox{.}
\end{split}\]
\end{lemma}
Based on the Lemma \ref{lemma_A}, the so-called \emph{UPS-factorization} was introduced. In addition, for the current article, we define also the function $\LUF(w)$ (the length of UPS-factorization).
\begin{definition}(see \cite[Definition $2$]{RukavickaRichWords2017}) 
%\label{dud87fejhk8}
We define UPS-factorization (Unioccurrent Palindromic Suffix factorization) to be the factorization of a finite nonempty rich word $w$ into the form $w=w_pw_{p-1}\cdots w_1$ from Lemma \ref{lemma_A}.
Let $\LUF(w)=p$ be \emph{the length of UPS-factorization} of the rich word $w$.
\end{definition}

Let $\rwc(n)$ denote the number of rich words of length $n$. In \cite{RukavickaRichWords2017}, the UPS-factorization has been used to prove that the function $\rwc(n)$ grows subexponentially; formally \begin{equation}\label{ud789eehrk9}\lim_{n\rightarrow\infty}(\rwc(n))^{\frac{1}{n}}=1\mbox{.}\end{equation} It was proved in \cite[Theorem $3$]{RukavickaRichWords2017}) that there is a positive real constant $c$ such that for every finite rich word $w$ with $n=\vert w\vert$ we have that \begin{equation}\label{dui8f9ejf}\LUF(w)\leq c\frac{n}{\ln{n}}\mbox{.}\end{equation} This can be interpreted as follows: A rich word $w$ is a concatenation of ``small'' number of palindromes, where ``small'' means that the length of UPS-factorization with regard to the length of the word $w$ tends to zero; formally we have that \begin{equation}\label{dyue789jd}\lim_{n\rightarrow\infty}\frac{c\frac{n}{\ln{n}}}{n}=\lim_{n\rightarrow\infty}\frac{c}{\ln{n}}=0\mbox{.}\end{equation} Since every palindrome is uniquely determined by its first half, the property (\ref{dyue789jd}) allowed in consequence to prove the subexponential growth of the number of rich words. 
%TODO purpose of better luf is to improve upper bound.

The formula (\ref{ud789eehrk9}) presents the currently best known upper bound on the number of rich words.
In \cite{GuShSh15} it was conjectured for rich words on the binary alphabet that for some infinitely growing function $g(n)$ the following holds true  \[{\rwc(n)} = \mathcal{O} \Bigl(\frac{n}{g(n)}\Bigr)^{\sqrt{n}}\mbox{.}\]

Next to UPS-factorization, another palindromic factorizations have been studied; see, for instance, \cite{Bannai2015,FrPuZa}. A \emph{palindromic length} $\PL(v)$ of the finite word $v$ is the minimal number of palindromes whose concatenation is equal to $v$. It was conjectured in 2013 that for every aperiodic infinite word $x$, the palindromic length of its factors is not bounded \cite{FrPuZa}. Despite a considerable effort the conjecture remains open. 

For every rich word $w$ we have obviously that $\PL(w)\leq \LUF(w)$. Thus the result of the current article presents also an upper bound on the palindromic length of rich words.
Note that, in general, for non-rich words, UPS-factorization does not need to exist.

The main result of the current article is the following theorem.
\begin{theorem}
\label{ididujjeoe90d}
For a given finite alphabet $\Alpha$, there are real positive constants $\mu, \pi$ such that if $w$ is a finite nonempty rich word over the alphabet $\Alpha$ and $n=\vert w\vert$ then \[\LUF(w)\leq \mu\frac{n}{e^{\pi\sqrt{\ln{n}}}}\mbox{.}\]
\end{theorem}
\begin{remark}
Note that the constants $\mu,\pi$ in Theorem \ref{ididujjeoe90d} and the constant $c$ in (\ref{dui8f9ejf}) depend on the size of the alphabet $\Alpha$.
\end{remark}

To prove Theorem \ref{ididujjeoe90d}, we use principally the same ideas like in \cite[Theorem $3$]{RukavickaRichWords2017}). The distinction is that in \cite[Theorem $3$]{RukavickaRichWords2017}) we considered all palindromes of given length factors of the rich word in question. In the current article we apply the upper bound for the factor complexity of rich words from \cite{RukavickaFacCmplx2021}. This allows the improvement of the upper bound (\ref{dui8f9ejf}).

Based on our research that we performed we present the following conjecture. 
\begin{conjecture}
For a given finite alphabet $\Alpha$, there is a positive real constant $\lambda$ such that if $w$ is a finite nonempty rich word over the alphabet $\Alpha$ and $n=\vert w\vert$ then $\LUF(w)\leq \lambda\sqrt{n}$.
\end{conjecture}

\section{Preliminaries}
Let $\mathbb{N}_1$ denote the set of all positive integers, let $\mathbb{R}$ denote the set of all real numbers, and let $\mathbb{R}^+$ denote the set of all positive real numbers. 

Let $\Alpha$ denote a finite alphabet and let $q=\vert \Alpha\vert$. Let $\Alpha^n$ denote the set of all words of length $n\in\mathbb{N}_1$ over the alphabet $\Alpha$. Let $\Alpha^+=\bigcup_{i\geq 1}\Alpha^i$, let $\epsilon$ denote the empty word, and let $\Alpha^*=\{\epsilon\}\cup\Alpha^+$. Let $\Alpha^{\infty}$ denote the set of all infinite words over the alphabet $\Alpha$; formally $\Alpha^{\infty}=\{a_1a_2\cdots \mid a_i\in\Alpha\mbox{ and }i\in\mathbb{N}_1\}$. 

Let $\Pal\subseteq\Alpha^*$ denote the set of all palindromes and let $\rws\subseteq \Alpha^*$ denote the set of all rich words. We have that $\epsilon\in\Pal\cap\rws$.

Let $\Factor_w\subseteq \Alpha^*$ be the set of all finite factors of the word $w\in\Alpha^*\cup\Alpha^{\infty}$. 
Let $\Factor_w(n)=\vert \Factor_w\cap\Alpha^n\vert$. The function $\Factor_w(n)$ is the \emph{factor complexity} of the word $w$.

An infinite word $w\in\Alpha^{\infty}$ is rich if $\Factor_w\subseteq \rws$; it means that all finite factors of $w$ are rich. Let $\rws^{\infty}\subseteq\Alpha^{\infty}$ denote the set of all infinite rich words.
\section{Factor complexity}

Let $\delta=\frac{3}{2(\ln{3}-\ln{2})}$. Recall that $q$ denote the size of the alphabet $\Alpha$. There is a subexponential upper bound on the factor complexity of rich words:
\begin{theorem} (see \cite[Theorem 1.1]{RukavickaFacCmplx2021})
\label{tjnmgh7ety}
If $w\in\rws\cup \rws^{\infty}$ and $n\in\mathbb{N}_1$ then
\[\vert \Factor_w(n)\vert \leq (4q^{2}n)^{\delta\ln{2n}+2}\mbox{.}\]
\end{theorem}
\begin{remark}
In \cite[Theorem 1.1]{RukavickaFacCmplx2021}, there is a condition that $q>1$. But the theorem obviously holds also for $q=1$, because in such a case we have that $\Factor_w(n)=1$ for every $n\in\mathbb{N}_1$. That is why we omit this condition in Theorem \ref{tjnmgh7ety}.
\end{remark}

Let $c_2=\delta(\ln{4}+2\ln{q})+\delta+\delta\ln{2}+2$. Let $\dd\in\mathbb{R}^+$ and let $\dd<1$. Let $\alpha=\sqrt{\frac{1}{c_2}}$ and let $\beta=\frac{-1-{\dd}}{2c_2}$. %TODO explain the constant \dd
Let $c_1$ be the minimal real constant such that $c_1\geq e^{(\delta\ln{2}+2)(\ln{4}+2\ln{q})}$ and $c_1e^{\beta+{\dd}\beta+c_2\beta^2}\geq 1$. It means that \[c_1=\min\{e^{(\delta\ln{2}+2)(\ln{4}+2\ln{q})}, e^{-(\beta+{\dd}\beta+c_2\beta^2)}\}\mbox{.}\]

\begin{remark}
The constant $0<\dd<1$ can be choosen arbitrarily.

The constants $\dd$, $\alpha$, and $\beta$ will be used in next sections. For Lemma \ref{y7xxcvch83g} it is necessary that $c_1e^{\beta+{\dd}\beta+c_2\beta^2}\geq 1$. That is why we introduced these constants already in this section, where $c_1$ is defined. 
\end{remark}

The merit of the next corollary is to have the upper bound for factor complexity of rich words in a simple form.
\begin{corollary}
\label{tture7eyu8ui}
If $w\in\rws\cup \rws^{\infty}$ and $n\in\mathbb{N}_1$ we have that \[\vert \Factor_w(n)\vert \leq c_1n^{c_2\ln{n}}\mbox{.}\]
\end{corollary}
\begin{proof}
We have that 
\[\begin{split}(4q^{2}n)^{\delta\ln{2n}+2}=e^{(\ln{4}+2\ln{q}+\ln{n})(\delta\ln{n}+\delta\ln{2}+2)}= \\ e^{\delta\ln{n}(\ln{4}+2\ln{q})+\delta\ln^2{n}+(\delta\ln{2}+2)(\ln{4}+2\ln{q})+(\delta\ln{2}+2)\ln{n}}\leq \\ e^{(\delta\ln{2}+2)(\ln{4}+2\ln{q})}e^{(\delta(\ln{4}+2\ln{q})+\delta+\delta\ln{2}+2)\ln^2{n}}=\\ e^{(\delta\ln{2}+2)(\ln{4}+2\ln{q})}n^{(\delta(\ln{4}+2\ln{q})+\delta+\delta\ln{2}+2)\ln{n}}\mbox{.}
\end{split}\]
Then the lemma follows from Theorem \ref{tjnmgh7ety}. This ends the proof.
\end{proof}

\section{UPS-length}

The next technical lemma shows a lower bound on the sum $\sum_{i=1}^{k}i\lfloor c_1i^{c_2\ln{i}}\rfloor$ using the constant $\dd$. 
The sum $\sum_{i=1}^{k}i\lfloor c_1i^{c_2\ln{i}}\rfloor$ has the following interpretation: It is the length of the word $w$ which is concatenation of $\lfloor c_1i^{c_2\ln{i}}\rfloor$ words of length $i$ for $i\in\{1,2,\dots,k\}$. 

\begin{lemma}
\label{djuyeim8445}
If $k\in\mathbb{N}_1$ then 
\[\sum_{i=1}^{k}i\lfloor c_1i^{c_2\ln{i}}\rfloor\geq (k^{\dd}-1)(k-k^{\dd})c_1(k-k^{\dd})^{c_2\ln{(k-k^{\dd})}}-\frac{k(k+1)}{2}\mbox{.}\]
\end{lemma}
\begin{proof}
We have that \[\begin{split}\sum_{i=1}^{k}i\lfloor c_1i^{c_2\ln{i}}\rfloor\geq\sum_{i=1}^{k}i(c_1i^{c_2\ln{i}}-1)=\sum_{i=1}^{k}ic_1i^{c_2\ln{i}}-\sum_{i=1}^{k}i = \\ \sum_{i=1}^{\lceil k-k^{\dd}\rceil}ic_1i^{c_2\ln{i}} + \sum_{i=\lceil k-k^{\dd}\rceil+1 }^{k}ic_1i^{c_2\ln{i}} -\sum_{i=1}^{k}i \geq \\ \sum_{i=\lceil k-k^{\dd}\rceil+1 }^{k}ic_1i^{c_2\ln{i}}-\sum_{i=1}^{k}i \geq \lfloor k^{\dd}\rfloor(k-k^{\dd})c_1(k-k^{\dd})^{c_2\ln{(k-k^{\dd})}}-\sum_{i=1}^{k}i\geq \\  (k^{\dd}-1)(k-k^{\dd})c_1(k-k^{\dd})^{c_2\ln{(k-k^{\dd})}} -\frac{k(k+1)}{2} \mbox{.}\end{split}\]
This completes the proof.
\end{proof}

For the sum $\sum_{i=1}^{k}i\lfloor c_1i^{c_2\ln{i}}\rfloor$ we consider $k$ to be a function of $n\in\mathbb{N}_1$. Then for the lower bound from Lemma \ref{djuyeim8445} we show a limit behavior as $n$ tends to infinity.
\begin{proposition}
\label{bb5f41441d}
If $n\in\mathbb{N}_1$, $\sigma:\mathbb{N}_1\rightarrow\mathbb{N}_1$, $\lim_{n\rightarrow\infty}\sigma(n)=\infty$, and $k_n=e^{\sigma(n)}$ then 
\[\lim_{n\rightarrow\infty}\frac{(k_n^{\dd}-1)(k_n-k_n^{\dd})c_1(k_n-k_n^{\dd})^{c_2\ln{(k_n-k_n^{\dd})}}-\frac{k_n(k_n+1)}{2}}{e^{(1+{\dd})\sigma(n)}c_1e^{c_2(\sigma(n))^2}}=1\mbox{.}\]
\end{proposition}
\begin{proof}
Let $\widetilde\sigma(n)=1-e^{({\dd}-1)\sigma(n)}$. We have that \[k_n-k_n^{\dd}=e^{\sigma(n)}-e^{{\dd}\sigma(n)}=e^{\sigma(n)}(1-e^{({\dd}-1)\sigma(n)})=e^{\sigma(n)}\widetilde\sigma(n)\mbox{.}\]
It follows that
\begin{equation}
\label{rrt489djvv1}
\begin{split}
(k_n^{\dd}-1)(k_n-k_n^{\dd})c_1(k_n-k_n^{\dd})^{c_2\ln{(k_n-k_n^{\dd})}}= \\ 
(e^{{\dd}\sigma(n)}-1)e^{\sigma(n)}\widetilde\sigma(n)c_1(e^{\sigma(n)}\widetilde\sigma(n))^{c_2\ln{(e^{\sigma(n)}\widetilde\sigma(n))}}=\\
(e^{{\dd}\sigma(n)}-1)e^{\sigma(n)}\widetilde\sigma(n)c_1(e^{\sigma(n)}\widetilde\sigma(n))^{c_2(\sigma(n)+\ln{\widetilde\sigma(n))}}=\\
(e^{{\dd}\sigma(n)}-1)e^{\sigma(n)}\widetilde\sigma(n)c_1e^{c_2(\sigma(n))^2}(\widetilde\sigma(n))^{c_2\sigma(n)}e^{c_2\sigma(n)\ln{\widetilde\sigma(n)}}(\widetilde\sigma(n))^{c_2\ln{\widetilde\sigma(n)}}
\mbox{.}
\end{split}
\end{equation}
We have that:
\begin{itemize}
\item
\begin{equation}\label{t72094mvf}\lim_{n\rightarrow\infty}\widetilde\sigma(n)=\lim_{n\rightarrow\infty}(1-e^{({\dd}-1)\sigma(n)})=1-0=1\mbox{.}\end{equation}
\item
\begin{equation}\label{ndsh762dfs}\begin{split}\lim_{n\rightarrow\infty}(\widetilde\sigma(n))^{\sigma(n)}=\lim_{n\rightarrow\infty}(1-e^{({\dd}-1)\sigma(n)})^{\sigma(n)}=\\ \lim_{n\rightarrow\infty}\left(1-\frac{1}{\left(e^{(1-{\dd})}\right)^{\sigma(n)}}\right)^{\sigma(n)}=1\mbox{.}\end{split}\end{equation}
\item
From (\ref{ndsh762dfs}) it follows that 
\begin{equation}\label{yd7623dhj}\lim_{n\rightarrow\infty}\sigma(n)\ln{\widetilde\sigma(n)}=\lim_{n\rightarrow\infty}\ln{\left(\widetilde\sigma(n)\right)^{\sigma(n)}}=\ln{1}=0\mbox{.}\end{equation}
\item
\begin{equation}\label{nb26ddfjfj}\begin{split}\lim_{n\rightarrow\infty}(\widetilde\sigma(n))^{\ln{\widetilde\sigma(n)}}=\lim_{n\rightarrow\infty}(1-e^{({\dd}-1)\sigma(n)})^{\ln{\sigma(n)}}=\\ \lim_{n\rightarrow\infty}\left(1-\frac{1}{\left(e^{(1-{\dd})}\right)^{\sigma(n)}}\right)^{\ln{\sigma(n)}}=1\mbox{.}\end{split}\end{equation}
\item
\begin{equation}\label{jhsd78djre}\lim_{n\rightarrow\infty}\frac{e^{{\dd}\sigma(n)}-1}{e^{{\dd}\sigma(n)}}=1\mbox{.}\end{equation}
\end{itemize}
From (\ref{t72094mvf}), (\ref{ndsh762dfs}), (\ref{yd7623dhj}), (\ref{nb26ddfjfj}), and (\ref{jhsd78djre}) it follows that 
\begin{equation}\label{ft66xvcb44}\lim_{n\rightarrow\infty}\frac{(e^{{\dd}\sigma(n)}-1)e^{\sigma(n)}\widetilde\sigma(n)c_1e^{c_2(\sigma(n))^2}(\widetilde\sigma(n))^{c_2\sigma(n)}e^{c_2\sigma(n)\ln{\widetilde\sigma(n)}}(\widetilde\sigma(n))^{c_2\ln{\widetilde\sigma(n)}}}{e^{(1+\dd)\sigma(n)}c_1e^{c_2(\sigma(n))^2}}=1\mbox{.}\end{equation}
Obviously 
\begin{equation}\label{dh887ejei9}\begin{split} \lim_{n\rightarrow\infty}\frac{\frac{k_n(k_n+1)}{2}}{e^{(1+{\dd})\sigma(n)}c_1e^{c_2(\sigma(n))^2}}=  \lim_{n\rightarrow\infty}\frac{e^{2\sigma(n)}+e^{\sigma(n)}}{2e^{(1+{\dd})\sigma(n)}c_1e^{c_2(\sigma(n))^2}}=0\mbox{.}\end{split}\end{equation}
The proposition follows from (\ref{rrt489djvv1}),  (\ref{ft66xvcb44}), and (\ref{dh887ejei9}). This ends the proof.
\end{proof}

For $\sigma(n)=\alpha\sqrt{\ln{n}}+\beta$ we show that the function $e^{(1+d)\sigma(n)}e^{c_2(\sigma(n))^2}$ used in Proposition \ref{bb5f41441d} is a multiple of $n$ by a constant.
\begin{lemma}
\label{d78edhhye55}
If $\sigma(n)=\alpha\sqrt{\ln{n}}+\beta$ and $n\in\mathbb{N}_1$ then \[\frac{e^{(1+{\dd})\sigma(n)}e^{c_2(\sigma(n))^2}}{n}=e^{\beta+{\dd}\beta+c_2\beta^2}\mbox{.}\]
\end{lemma}
\begin{proof}

We have that 
\begin{equation}\label{vvnd783hdj}\begin{split}\frac{e^{(1+{\dd})\sigma(n)}e^{c_2(\sigma(n))^2}}{n}=e^{(1+{\dd})(\alpha\sqrt{\ln{n}}+\beta)+c_2(\alpha\sqrt{\ln{n}}+\beta)^2-\ln{n}}
\mbox{}\end{split}\end{equation}
and 
\begin{equation}\label{v65sh29kjh}\begin{split}(1+{\dd})(\alpha\sqrt{\ln{n}}+\beta)+c_2(\alpha\sqrt{\ln{n}}+\beta)^2-\ln{n}=\\ \alpha\sqrt{\ln{n}}+\beta+{\dd}\alpha\sqrt{\ln{n}}+{\dd}\beta+c_2\alpha^2\ln{n}+c_22\alpha\beta\sqrt{\ln{n}}+c_2\beta^2-\ln{n}=\\  (c_2\alpha^2-1)\ln{n}+\alpha(1+{\dd}+2c_2\beta)\sqrt{\ln{n}}+\beta+{\dd}\beta+c_2\beta^2=\\  (c_2\left(\sqrt{\frac{1}{c_2}}\right)^2-1)\ln{n}+\alpha(1+{\dd}+2c_2\frac{-1-{\dd}}{2c_2})\sqrt{\ln{n}}+\beta+{\dd}\beta+c_2\beta^2=\\ \beta+{\dd}\beta+c_2\beta^2\mbox{.}\end{split}\end{equation}

The proposition follows from (\ref{vvnd783hdj}) and (\ref{v65sh29kjh}).
This ends the proof.
\end{proof}

Using Lemma \ref{djuyeim8445},  Lemma \ref{d78edhhye55}, and Proposition \ref{bb5f41441d} we show that the sum $\sum_{i=1}^{k_n}i\lfloor c_1i^{c_2\ln{i}}\rfloor$ is bigger than $n$ as $n$ tends to infinity.
\begin{corollary}
\label{y7xxcvch83g}
If $\sigma(n)=\alpha\sqrt{\ln{n}}+\beta$ and $k_n=e^{\sigma(n)}$ then 
\[\lim_{n\rightarrow\infty}\frac{\sum_{i=1}^{\lfloor k_n\rfloor}i\lfloor c_1i^{c_2\ln{i}}\rfloor}{n}\geq c_1e^{\beta+{\dd}\beta+c_2\beta^2}\geq 1\mbox{.}\]
\end{corollary}
\begin{proof}
Let $\widetilde k_n=\lfloor k_n\rfloor$. Obviously we have that
\begin{equation}\label{bf4r56465}\lim_{n\rightarrow}\frac{(\widetilde k_n^{\dd}-1)(\widetilde k_n-\widetilde k_n^{\dd})c_1(\widetilde k_n-\widetilde k_n^{\dd})^{c_2\ln{(\widetilde k_n-\widetilde k_n^{\dd})}}-\frac{\widetilde k_n(\widetilde k_n+1)}{2}}{(k_n^{\dd}-1)(k_n-k_n^{\dd})c_1(k_n-k_n^{\dd})^{c_2\ln{(k_n-k_n^{\dd})}}-\frac{k_n(k_n+1)}{2}}=1\mbox{.}\end{equation}
Then Lemma \ref{djuyeim8445}, Proposition \ref{bb5f41441d}, and (\ref{bf4r56465}) imply that 
\begin{equation}\label{vnzoid66e59}\lim_{n\rightarrow\infty}\frac{\sum_{i=1}^{\widetilde k_n}i\lfloor c_1i^{c_2\ln{i}}\rfloor}{e^{(1+{\dd})\sigma(n)}c_1e^{c_2(\sigma(n))^2}}\geq 1\mbox{.}\end{equation}

Recall that from the definition of $c_1$ we have that $c_1e^{\beta+{\dd}\beta+c_2\beta^2}\geq 1$.
Hence Lemma \ref{d78edhhye55} and (\ref{vnzoid66e59}) imply the corollary.
This completes the proof.
\end{proof}

Let $\gamma=\beta+\ln{c_1}+c_2\beta^2=\frac{-1-{\dd}}{2c_2}+\ln{c_1}+c_2(\frac{-1-{\dd}}{2c_2})^2\in\mathbb{R}$. The next lemma shows an upper bound for the sum $\sum_{i=1}^{k_n}c_1i^{c_2\ln{i}}$, where $k_n=\lfloor e^{\alpha\sqrt{\ln{n}}+\beta}\rfloor$.
\begin{lemma}
\label{ooi22u437gdt5}
If $n\in\mathbb{N}_1$ and $\sigma(n)=\alpha\sqrt{\ln{n}}+\beta$ then 
\[\begin{split}\sum_{i=1}^{\lfloor e^{\sigma(n)}\rfloor}c_1i^{c_2\ln{i}}\leq e^{\gamma}\frac{n}{e^{{\dd}\alpha\sqrt{\ln{n}}}}\mbox{.} 
\end{split}\]
\end{lemma}
\begin{proof}
We have that 
\begin{equation}\label{vn9837hdkw1}\begin{split}\sum_{i=1}^{\lfloor e^{\sigma(n)}\rfloor}c_1i^{c_2\ln{i}}\leq e^{\sigma(n)}c_1\left(e^{\sigma(n)}\right)^{c_2\ln{e^{\sigma(n)}}}= e^{\sigma(n)}c_1e^{c_2(\sigma(n))^2}=\\ e^{\alpha\sqrt{\ln{n}}+\beta}e^{\ln{c_1}}e^{c_2(\alpha\sqrt{\ln{n}}+\beta)^2}\mbox{}\end{split}\end{equation}
and 
\begin{equation}\label{vv78zz1po}\begin{split} \alpha\sqrt{\ln{n}}+\beta+\ln{c_1}+c_2(\alpha\sqrt{\ln{n}}+\beta)^2=\\ \alpha\sqrt{\ln{n}}+\beta+\ln{c_1}+c_2\alpha^2\ln{n}+c_22\alpha\beta\sqrt{\ln{n}}+c_2\beta^2=\\ 
c_2\alpha^2\ln{n}+\sqrt{\ln{n}}(\alpha+2c_2\alpha\beta)+\gamma\mbox{.}\end{split}\end{equation}
Note that $c_2\alpha=c_2(\sqrt{\frac{1}{c_2}})^2=1$ and $\alpha+2c_2\alpha\beta=\sqrt{\frac{1}{c_2}}(1+2c_2\frac{-1-{\dd}}{2c_2})=-{\dd}\alpha$. Then it follows from (\ref{vn9837hdkw1}) and (\ref{vv78zz1po}) that 
\[\sum_{i=1}^{\lfloor e^{\sigma(n)}\rfloor}c_1i^{c_2\ln{i}}\leq e^{\ln{n}+\sqrt{\ln{n}}(-{\dd}\alpha)+\gamma}=e^{\gamma}\frac{n}{e^{{\dd}\alpha\sqrt{\ln{n}}}}\mbox{.}\]
This completes the proof.
\end{proof}

In the proof of the main theorem we will need the following technical lemma. 
\begin{lemma}
\label{d78ejjhd111w}
If $\overline t,t\in\mathbb{N}_1$, $\overline t>t$, $\overline\omega, \omega:\mathbb{N}_1\rightarrow\mathbb{N}_1$, $\overline\omega(i)\leq \omega(i)$ for every $i\in\mathbb{N}_1$, and \[\sum_{i=1}^{\overline t}\overline\omega(i)>\sum_{i=1}^{t}\omega(i)\] then 
\[\sum_{i=1}^{\overline t}i\overline\omega(i)>\sum_{i=1}^{t}i\omega(i)\mbox{.}\]
\end{lemma}
\begin{proof}
We have that 
\begin{equation}\label{g77e8vxsh}\begin{split}\sum_{i=1}^{\overline t}\overline\omega(i)-\sum_{i=1}^{t}\omega(i)=\\ \sum_{i=1}^{t}\overline\omega(i)+\sum_{i=t+1}^{\overline t}\overline\omega(i) - \left(\sum_{i=1}^{t}(\omega(i)-\overline \omega(i))+\sum_{i=1}^{t}\overline\omega(i)\right) = \\ \sum_{i=t+1}^{\overline t}\overline\omega(i)-\sum_{i=1}^{t}(\omega(i)-\overline \omega(i))\mbox{.}\end{split}\end{equation}
Since $\sum_{i=1}^{\overline t}\overline\omega(i)>\sum_{i=1}^{t}\omega(i)$ it follows from (\ref{g77e8vxsh}) that \begin{equation}\label{ttncb129c}\sum_{i=t+1}^{\overline t}\overline\omega(i)>\sum_{i=1}^{t}(\omega(i)-\overline \omega(i))\mbox{.}\end{equation}

From (\ref{ttncb129c}) it is obvious that \begin{equation}\label{du739nbsj9}\sum_{i=t+1}^{\overline t}i\overline\omega(i)>\sum_{i=1}^{t}i(\omega(i)-\overline \omega(i))\mbox{.}\end{equation}
We have that
\begin{equation}\label{g88s22bcgf}\begin{split}\sum_{i=1}^{\overline t}i\overline\omega(i)-\sum_{i=1}^{t}i\omega(i)=\\ \sum_{i=1}^{t}i\overline\omega(i)+\sum_{i=t+1}^{\overline t}i\overline\omega(i) - \left(\sum_{i=1}^{t}i(\omega(i)-\overline \omega(i))+\sum_{i=1}^{t}i\overline\omega(i)\right) = \\ \sum_{i=t+1}^{\overline t}i\overline\omega(i)-\sum_{i=1}^{t}i(\omega(i)-\overline \omega(i))\mbox{.}\end{split}\end{equation}
It follows from (\ref{du739nbsj9}) and (\ref{g88s22bcgf}) that $\sum_{i=1}^{\overline t}i\overline\omega(i)>\sum_{i=1}^{t}i\omega(i)$.
This ends the proof.
\end{proof}

\begin{remark}
Lemma \ref{d78ejjhd111w} can be interpreted as follows.
Suppose two finite sets $\overline \Omega,\Omega\subset \Alpha^+$. Let $\overline \omega(i)=\vert \overline\Omega\cap\Alpha^i\vert$ and $\omega(i)=\vert \Omega\cap\Alpha^i\vert$. Let $\overline w, w\in\Alpha^+$ be words that are a concatenation of all words from $\overline\Omega, \Omega$ respectively. Lemma \ref{d78ejjhd111w} implies that if $\vert\overline\Omega\vert>\vert\Omega\vert$ and $\overline \omega(i)\leq\omega(i)$ then $\vert \overline w\vert>\vert w\vert$. This interpretation will simplify the comprehension of the proof of Proposition \ref{tued8jnnsjhs}.
\end{remark}

The next proposition shows a technique how to derive an upper bound for the length of UPS-factorization.
\begin{proposition}
\label{tued8jnnsjhs}
If $w\in\rws\cap\Alpha^+$, $n=\vert w\vert$, $t\in\mathbb{N}_1$, and \[\sum_{i=1}^t i\lfloor c_1i^{c_2\ln{i}}\rfloor\geq n\] then \[\LUF(w)\leq \sum_{i=1}^t c_1i^{c_2\ln{i}}\mbox{.}\]
\end{proposition}
\begin{proof}
Let $t\in\mathbb{N}_1$ be the minimal value such that \begin{equation}\label{dju87883jk}\sum_{i=1}^t i\lfloor c_1i^{c_2\ln{i}}\rfloor\geq n\mbox{.}\end{equation}
We claim that $\LUF(w)\leq \sum_{i=1}^t c_1i^{c_2\ln{i}}$. To get a contradiction suppose that \begin{equation}\label{dhy888djhw7}h=\LUF(w)>\sum_{i=1}^t c_1i^{c_2\ln{i}}\mbox{.}\end{equation} 
Let $w_1,w_2,\dots,w_h\in\Factor_w\cap \Pal\cap\Alpha^+$ be distinct nonempty palindromes of UPS-factorization of $w$; it means that $w=w_{1}w_{2}\cdots w_{h}$.
Let \[\phi(k)=\vert\{w_i\mid w_i\in\Alpha^k\mbox{ and }i\in\{1,2,\dots,h\}\}\vert\] be the number of palindromes $w_i$ of length $k$.
Let $\overline t\in\mathbb{N}_1$ be the maximal value such that $\phi(\overline t)>0$. Clearly we have that $h=\sum_{k= 1}^{\overline t}\phi(k)$. Hence it follows from (\ref{dhy888djhw7}) that \begin{equation}\label{djue8r7jj}\sum_{k= 1}^{\overline t}\phi(k)>\sum_{i=1}^t c_1i^{c_2\ln{i}}\mbox{.}\end{equation}

Since $w$ is rich, Corollary \ref{tture7eyu8ui} implies that \begin{equation}\label{duu77c6sj2bdj3}\phi(k)\leq c_1k^{c_2\ln{k}}\quad\mbox{ for every }k\in\{1,2,\dots,\overline t\}\mbox{.}\end{equation}

Suppose that $\overline t>t$. Then from Lemma \ref{d78ejjhd111w}, (\ref{djue8r7jj}), and (\ref{duu77c6sj2bdj3}) it follows that \[\sum_{k= 1}^{\overline t}k\phi(k)>\sum_{i=1}^t ic_1i^{c_2\ln{i}}\mbox{,}\] which is a contradiction to (\ref{dju87883jk}), since obviously $n=\sum_{k= 1}^{\overline t}k\phi(k)$. We conclude that $\overline t\leq t$. Then (\ref{djue8r7jj}) implies that there is at least one $k$ such that $\phi(k)> c_1k^{c_2\ln{k}}$, which contradicts to (\ref{duu77c6sj2bdj3}).
We conclude that $\LUF(w)\leq \sum_{i=1}^t c_1i^{c_2\ln{i}}$. This completes the proof.
\end{proof}

Now, we can step to the  proof of the main theorem of the current article.
\begin{proof}[Proof of Theorem \ref{ididujjeoe90d}]
We have that $w\in\rws\cap\Alpha^+$ and $n=\vert w\vert$.
Let \[t= \lfloor e^{\alpha\sqrt{\ln{n}}+\beta}\rfloor\mbox{.}\] Then Corollary \ref{y7xxcvch83g} implies that \begin{equation}\label{dbh37djhdu}\sum_{i=1}^t i\lfloor c_1i^{c_2\ln{i}}\rfloor\geq n\end{equation} and Lemma \ref{ooi22u437gdt5} implies that \begin{equation}\label{d66eh788d73j}\sum_{i=1}^t c_1i^{c_2\ln{i}}\leq e^{\gamma}\frac{n}{e^{{\dd}\alpha\sqrt{\ln{n}}}}\quad\mbox{as $n$ tends to infinity.}\end{equation}  
Proposition \ref{tued8jnnsjhs} and (\ref{dbh37djhdu}) imply that 
\begin{equation}\label{tue7duyubb2}\LUF(w)\leq\sum_{i=1}^t c_1i^{c_2\ln{i}}\mbox{.}\end{equation}
From (\ref{d66eh788d73j}) and (\ref{tue7duyubb2})  it follows that there are constants $\mu,\pi\in\mathbb{R^+}$ such that \[\LUF(w)\leq \mu\frac{n}{e^{\pi\sqrt{\ln{n}}}}\mbox{.}\]
This completes the proof.
\end{proof}

%%%%%%%%%%%%%%%%%%%%%%%%%%%%%%%%%%%%%%%%%%%
\section*{Acknowledgments}
This work was supported by the Grant Agency of the Czech Technical University in Prague, grant No. SGS20/183/OHK4/3T/14.
%The author acknowledges support by the Czech Science
%Foundation grant GA\v CR 13-03538S and by the Grant Agency of the Czech Technical University in Prague, grant No. SGS14/205/OHK4/3T/14.

\bibliographystyle{siam}
\IfFileExists{biblio.bib}{\bibliography{biblio}}{\bibliography{../!bibliography/biblio}}

\end{document}